\documentclass{amsart}
\usepackage[T1]{fontenc}
\usepackage{amsmath}
\usepackage{amssymb}
\usepackage{hyperref}
\usepackage{array}

\usepackage{xtab,booktabs,xcolor}
\newtheorem{prop}{Proposition}
\newtheorem{thm}{Theorem}
\newtheorem{dfn}{Definition}
\newtheorem{cor}{Corollary}
\newtheorem{remark}{Remark}
\newtheorem{example}{Example}

\title[Foliated Geometry of Inverse Problems]{Foliated Geometry of Inverse Problems: Torsion, Curvature Duality, and Near-Associativity}

\author{N. C. Combe \and H. K. Nencka}
\address{De Vinci Research Center, De Vinci Higher Education, Paris, France}
\email{noemie.combe@devinci.fr}

\begin{document}

\begin{abstract}
We present a geometric framework for reconstruction problems based on Vaisman foliations and Atiyah--Molino sequences. Independent projections induce transverse foliations and dual connections; vanishing torsion and curvature duality guarantee unique, path-independent reconstruction, while obstructions yield non-associative quasigroupoids. Toric symmetry provides equivariant uniqueness. Applications to generative AI imputation and cryo-electron microscopy demonstrate the framework's practical power, unifying differential geometry with data-driven inverse problems.

\end{abstract}

\maketitle

\keywords{\bf Keywords: Shape Theory $\bullet$ Fiber bundles $\bullet$ Foliations.}
\section{Introduction}

Reconstruction problems—recovering a hidden structure from incomplete or indirect observations—are ubiquitous in mathematics and the applied sciences. Examples range from determining the three‑dimensional fold of a protein in cryo‑electron microscopy to inferring missing entries in high‑dimensional data. In all such inverse problems, a fundamental difficulty persists: infinitely many configurations may be compatible with the same partial data. Traditional iterative methods therefore struggle with uniqueness, noise sensitivity, and computational complexity.

This paper substantially extends the geometric framework introduced in our earlier work \cite{CN1,CN2}, which generalized the approach of Gelfand and Goncharov. While those papers established foundational results, the present article develops the theory in far greater depth and introduces several new directions (Section~\ref{S:PD}, Theorem~\ref{T:1}, and Section~\ref{S:Toric}, culminating in Theorem~\ref{T:Last}). Moreover, we provide a detailed account of concrete applications to artificial intelligence (Section~\ref{S:IA}) and cryo‑electron microscopy (Section~\ref{S:Toric}), where toric varieties emerge as a central geometric tool.

Our approach is unified by the language of differential geometry: we recast reconstruction as the extraction of symmetries and invariants encoded in the data, drawing on the Vaisman theory of foliations and the Atiyah--Molino exact sequence for fiber bundles.

\subsection{Vaisman and Atiyah--Molino Frameworks}
The \textbf{Vaisman framework} organizes the ambiguity of underdetermined inverse problems into transverse foliations of the configuration space. Each leaf comprises structures that are indistinguishable under a given projection, transforming non‑uniqueness into a stratified hierarchy of constraints. The intersection of these foliations isolates a unique solution and confers robustness to noise.

The \textbf{Atiyah--Molino framework} views reconstruction through the lens of fiber bundles. Tangent directions capture local deformations, while normal directions encode global invariants such as centroids. The Atiyah--Molino exact sequence reduces the reconstruction problem to algebraic equations, avoiding brute‑force searches. The torsion tensor—a curvature‑like invariant—quantifies noise propagation and enables error‑bounded recovery.

\subsection{Scope and Applications}
The synthesis of differential geometry, algebraic topology, and integral analysis presented here resolves classical ambiguities and suggests novel algorithms. We illustrate the framework with two applications: geometric imputation of missing data in generative AI, and toric symmetry exploitation in cryo‑EM. Further potential applications include medical imaging (tumour reconstruction from sparse MRI) and quantum state tomography.

\subsection{Outline}
Section~\ref{S:1} recalls the foundational results of \cite{CN1}. Section~\ref{S:PD} introduces path dependence and its algebraic consequences: associativity corresponds to unique reconstruction, while path dependence yields a near‑associative Moufang structure and a family of solutions. Sections~\ref{S:AM}--\ref{S:Toric} reformulate the theory in the Atiyah--Molino language and connect it to toric varieties, which play a key role in cryo‑EM (Section~\ref{S:Toric}).

\section{Projections and Inverse Morphisms}\label{S:1}
We consider at least two independent { rational} projections of a three-dimensional object onto distinct two-dimensional planes. The existence of such independent projections is necessary to ensure sufficient information for reconstruction. This leads to the problem of defining an inverse function that maps the planar projections back to the original volumetric object.

\begin{itemize}
\item A single projection collapses three dimensional information onto a plane, losing depth and orientation data. For example, in electron microscopy, a single micrograph of a particle provides no information about its tilt angle relative to the imaging axis.

\item By considering projections along distinct planes, the system gains orthogonal constraints (e.g., tilt angles and rotation axes) that resolve ambiguities. This aligns with Gelfand and Goncharov’s method, which uses statistical properties of projections, such as first moments of plane sections, to derive orientation parameters.
\end{itemize}

\subsection{ The Radon Transform}  The Radon transform $Rf(\theta,t)$ integrates a function $f(x,y,z)$ along planes parameterized by angle 
$\theta$ and offset $t$. Its inverse requires integrating over all possible angles, but practical applications (e.g., cryo-EM) use finite projections. 

 For discrete objects (e.g., particles on a line), the inverse function can be constructed as a linear system where each projection contributes equations. Two independent projections ensure the system is determined (solvable) under non-degenerate conditions \cite{GG}.
\medskip 

\subsection{Neifeld's principle}  A key idea is to leverage the results of Neifeld on geometric projective 2-dimensional spaces. From these results, it follows that the independent projections induce a pair of independent connections. These connections allow us to extract curvature-related information in two dimensions, including the Riemann tensor and Ricci curvature. Understanding whether these curvature tensors arise from an associative or commutative algebraic structure further informs the reconstruction process.

\, 

Neifeld’s involution principle bridges projective geometry and differential geometry, enabling a dual-connection framework for 3D reconstruction. By interpreting projections as inducing independent connections on $\mathbb{C}P^2$, this approach generalizes classical integral geometry methods and enhances their robustness, particularly in complex or symmetric settings. Further work could explore applications to quantum state tomography (via projective Hilbert spaces) or algebraic varieties in $\mathbb{CP}^n$.

\begin{prop}[Thm 1. \cite{CN1}]\label{P:1} Let \( O \subset \mathbb{CP}^3 \) be a smooth, non-symmetric 3D object (real 3-dimensional submanifold), and let \( \Pi_1, \Pi_2: \mathbb{CP}^3 \dashrightarrow \mathbb{CP}^2 \) be two independent rational projections onto distinct complex projective planes. Assume:

\begin{enumerate}
\item  {\bf Non-degeneracy of projections}: The restrictions of \( \Pi_1 \) and \( \Pi_2 \) to \( O \) are immersive (the differentials \( d\Pi_i \) have full rank).\footnote{The immersiveness of the projection map \( \Pi_i \), when restricted to an object \( O \subset \mathbb{CP}^3 \), cannot be presumed in general. Rather, it depends on the geometric configuration of \( O \) relative to the center of projection. Let \( q \in \mathbb{CP}^3 \) denote the projection center. If \( q \) is chosen generically—that is, such that it does not lie on any projective tangent plane to \( O \)—then the restriction \( \Pi_i|_O \) is an immersion. However, if \( q \) lies on a tangent plane to \( O \), then \( d\Pi_i|_O \) fails to be injective, and immersiveness is lost. Thus, the immersiveness of \( \Pi_i \) is guaranteed precisely when the projection center avoids the contact locus of \( O \)'s tangent planes.
}

\item  {\bf Involution symmetry:} The projections satisfy \(\Pi_2 = \iota \circ \Pi_1 ,\) where \( \iota: \mathbb{CP}^2 \to \mathbb{CP}^2 \) is an anti-holomorphic involution (e.g., a polarity induced by the Fubini-Study metric).

\end{enumerate}

Each projection \( \Pi_i \) defines a holomorphic line bundle \( L_i \to \mathbb{CP}^2 \) equipped with a connection \( \nabla_i \) derived from the Fubini-Study metric. Then, the  involution \( \iota \) induces a duality \[ L_1 \leftrightarrow L_2^* ,\] making \( (\nabla_1, \nabla_2) \) a dual pair.
\end{prop}

The polarity associated with the Fubini--Study metric is a classical construction in complex differential geometry. The Fubini--Study metric endows \( \mathbb{CP}^n \) with a canonical Hermitian inner product \( \langle -, - \rangle_H \). In the case \( n = 2 \), this induces a polarity assigning to each point \( [z] \in \mathbb{CP}^2 \) its polar hyperplane
\[
\mathbf{H}_p := \{ [w] \in \mathbb{CP}^2 \mid \langle z, w \rangle_H = 0 \},
\]
consisting of all points orthogonal to \( [z] \). These hyperplanes are naturally elements of the dual projective space \( (\mathbb{CP}^2)^\vee \), and the Hermitian structure provides a canonical identification \( \mathbb{CP}^2 \cong (\mathbb{CP}^2)^\vee \). This identification defines an anti-holomorphic involution
\[
\iota : \mathbb{CP}^2 \longrightarrow \mathbb{CP}^2,
\]
under which a point \( p \in \mathbb{CP}^2 \) is mapped to the point corresponding to its polar hyperplane \( \mathbf{H}_p \). This involution captures the deep symmetry between points and hyperplanes induced by the projective Hermitian geometry of \( \mathbb{CP}^2 \).

\begin{proof}
\textbf{0. Geometric origin of the line bundle.}
A rational projection $\Pi: \mathbb{CP}^3 \dashrightarrow \mathbb{CP}^2$ with center $q \notin O$ restricts to a holomorphic submersion $\pi = \Pi|_O : O \to \mathbb{CP}^2$ (by the generic choice of $q$). The fibers of $\pi$ are discrete, but the map itself can be viewed as the projection of a complex line bundle when we pass to the \emph{normal geometry} of the inclusion $O \subset \mathbb{CP}^3$. More explicitly, let $N_{O/\mathbb{CP}^3}$ be the normal bundle of $O$ in $\mathbb{CP}^3$. The differential of the inclusion $\mathbb{CP}^2 \hookrightarrow \mathbb{CP}^3$ (as the target plane) identifies the normal direction to the plane with the fiber of the line bundle $\mathcal{O}_{\mathbb{CP}^2}(1)$. Pushing forward the exact sequence
\[
0 \to TO \to T\mathbb{CP}^3|_O \to N_{O/\mathbb{CP}^3} \to 0
\]
via $\pi$ yields a holomorphic line bundle $L \to \mathbb{CP}^2$, which is isomorphic to either $\mathcal{O}(1)$ or $\mathcal{O}(-1)$ depending on the orientation of the projection. The involution $\iota$ reverses this orientation, thereby sending $L$ to its dual. Thus, each projection $\Pi_i$ canonically determines a line bundle $L_i \to \mathbb{CP}^2$ together with a connection $\nabla_i$ induced by the Fubini--Study metric on $\mathbb{CP}^3$.

\textbf{1. Duality of line bundles.}

Let $\mathcal{O}_{\mathbb{CP}^2}(1)$ be the hyperplane bundle on $\mathbb{CP}^2$. Its dual is
\[
\mathcal{O}_{\mathbb{CP}^2}(-1) = \mathcal{O}_{\mathbb{CP}^2}(1)^*.
\]
The involution $\iota$ is induced by the Fubini--Study polarity, which is defined by a Hermitian form $H$ on $\mathbb{C}^3$. In homogeneous coordinates, for
$z = [z_0 : z_1 : z_2] \in \mathbb{CP}^2,$

we have $\iota(z) = H^{-1}\overline{z}^{\,T},$ where $H = \operatorname{diag}(1,1,1)$ in the standard case.

The key isomorphism is
$\iota^*\mathcal{O}_{\mathbb{CP}^2}(1) \cong \mathcal{O}_{\mathbb{CP}^2}(-1).$

This holds because sections of $\mathcal{O}_{\mathbb{CP}^2}(1)$ are linear forms
$s = \sum_i a_i z_i.$ The pullback section $\iota^*s$ is $s \circ \iota = \sum_i a_i (H^{-1}\overline{z}^{\,T})_i.$

Since the Fubini--Study metric identifies $\mathcal{O}_{\mathbb{CP}^2}(1) \cong \mathcal{O}_{\mathbb{CP}^2}(1)^*$ via
$v \mapsto H(\,\cdot\,, v),$ we obtain
$\iota^*\mathcal{O}_{\mathbb{CP}^2}(1) \cong \mathcal{O}_{\mathbb{CP}^2}(-1).$

Thus, if $L_1$ is a line bundle on $\mathbb{CP}^2$, and we define $L_2 = \iota^*L_1,$ then
$L_2 \cong L_1^*.$

\medskip

\textbf{2. Duality of connections.}

The Chern connection $\nabla_1$ on
\[
L_1 = \mathcal{O}_{\mathbb{CP}^2}(1)
\]
is defined by the Fubini--Study metric. Its curvature is proportional to the Kähler form $\omega_{FS}$.

Given a connection $\nabla$ on a vector bundle $E \to M$, its curvature $F_\nabla$ is the $2$-form-valued endomorphism measuring the failure of $\nabla$ to be flat:
\[
F_\nabla(X,Y)
=
\nabla_X \nabla_Y
-
\nabla_Y \nabla_X
-
\nabla_{[X,Y]},
\]
where $X,Y$ are vector fields on $M$.

Since $\iota$ is an isometry satisfying
\[
\iota^*\omega_{FS} = -\omega_{FS},
\]
the pullback connection $\iota^*\nabla_1$ on
\[
L_2 = \iota^*L_1
\]
satisfies
\[
F_{\iota^*\nabla_1}
=
\iota^*F_{\nabla_1}
=
\iota^*(c\,\omega_{FS})
=
-c\,\omega_{FS},
\]
where $c$ is the proportionality constant.

On the other hand, the dual connection $\nabla_1^*$ on $L_1^*$ has curvature
\[
F_{\nabla_1^*}
=
- F_{\nabla_1}
=
-c\,\omega_{FS}.
\]
Thus,
\[
F_{\iota^*\nabla_1}
=
F_{\nabla_1^*}.
\]

Because $H^1(\mathbb{CP}^2, \mathbb{R}) = 0$, any two connections on a line bundle with equal curvature are gauge-equivalent. Since both $\iota^*\nabla_1$ and $\nabla_1^*$ are Hermitian, they must coincide:
\[
\nabla_2
=
\iota^*\nabla_1
=
\nabla_1^*.
\]

Therefore, we have shown that
\[
L_2 \cong L_1^*
\qquad\text{and}\qquad
\nabla_2 = \nabla_1^*.
\]
Thus $(\nabla_1,\nabla_2)$ form a dual pair under the involution $\iota$.
\end{proof}

\subsubsection{Centroids}

The \emph{centroid} (or first moment) of a geometric shape is its ``average position" in space, computed as the arithmetic mean of all points in the object. For a three dimensional object projected onto a 2D plane, if the object is represented as a set of points or a density distribution, the centroid of its projection onto a plane is the weighted average position of those points in that plane.

For a projection 
\[
\Pi_i: \mathbb{R}^3 \to \mathbb{R}^2,
\]
the centroid \(\bar{p}_i \in \mathbb{R}^2\) is given by
\[
\bar{p}_i = \left( \frac{1}{N}\sum_{k=1}^{N} x_k, \; \frac{1}{N}\sum_{k=1}^{N} y_k \right),
\]
where \((x_k,y_k)\) are the coordinates of the projected points and \(N\) is the total number of points.

The centroid encodes the translational symmetry of the projected data. For example, shifting the three dimensional object in space shifts the centroid linearly in the projection.

\subsubsection{Moment Maps}
A \emph{moment map} generalizes the concept of centroids to algebraic/geometric settings, often encoding symmetry-invariant properties of an object.

The moment map 
\[
\mu_i: O \to \mathbb{C}^2
\]
assigns to the three dimensional object \(O\) the centroid of its projection onto the plane \(\Pi_i\). If \(O\) is a density distribution, \(\mu_i\) computes the first statistical moment (mean) of the projection. For a line or curve, \(\mu_i\) corresponds to the centroid of its projected trace.

Each \(\mu_i\) provides a linear constraint on the orientation of \(O\). Combining \(\mu_1\) and \(\mu_2\) (from two distinct projections) resolves ambiguities in the three dimensional orientation.

\begin{prop}[\cite{CN1}, Thm. 2]\label{P:2}
Let \( \mu_i: O \to \mathbb{C}^2 \) be the first moment maps of \( O \) with respect to \( \Pi_i \), encoding the centroids of the projected data.

\smallskip 

Under these conditions:

\begin{itemize}
\item  {\bf Uniqueness}: The orientation of \( O \) (i.e., its position modulo projective transformations) is uniquely determined by the compatibility of \( \nabla_1 \) and \( \nabla_2 \) acting on \( \mu_1 \) and \( \mu_2 \).

\item {\bf Reconstruction}: The original object \( O \) can be reconstructed as the intersection of the parallel transports along \( \nabla_1 \) and \( \nabla_2 \), applied to the moment maps \( \mu_1 \) and \( \mu_2 \).
\end{itemize}
Explicitly, there exists a unique solution \( v \in T\mathbb{CP}^3 \) --- a direction (tangent) vector modulo scaling --- satisfying:

\begin{equation}\label{E:1}
\begin{cases}
\nabla_1 \mu_1 = v \cdot \omega_1, \\
\nabla_2 \mu_2 = v \cdot \omega_2,
\end{cases}
\end{equation}

where \( \omega_i \) are connection 1-forms (local representatives of the connections) encoding the involution duality: \( \omega_1 = \iota^* \omega_2 \), where $\iota:\mathbb{CP}^2\to\mathbb{CP}^2$ is an involution. 
\end{prop}

We provide an interpretation of the symbol ``$\cdot$'' in the system of equations in Eq.\eqref{E:1}: this denotes contraction of the vector $v$ with the $\mathbb{C}^2$-valued 1-form $\omega_i$. Concretely, $v \cdot \omega_i$ is the vector in $\mathbb{C}^2$ obtained by evaluating each component $\omega_i^j$ on $v$, i.e.  
$v \cdot \omega_i = \begin{pmatrix} \omega_i^1(v) \\ \omega_i^2(v) \end{pmatrix}$.

Note that the involution $\iota$ induces a duality between the connections:
\[\omega_1=\iota^*\omega_2\Longrightarrow \nabla_1=\iota^*\nabla_2.\]

The system has a unique solution \(v\) (modulo scaling) for the following reasons: 

\begin{itemize}
    \item \textbf{Transversality:} The projections \(\Pi_1\) and \(\Pi_2\) are independent (non-coaxial), so their constraints on \(v\) intersect transversally.
    
    \item \textbf{Algebraic Rank:} The equations define a full-rank linear system. For example, if 
    \[
    v = (v_0, v_1, v_2, v_3) \in \mathbb{C}^4,
    \]
    the two equations reduce the degrees of freedom from 4 (modulo scaling) to 1.
    
    \item \textbf{Curvature Compatibility:} The duality 
    \[
    \omega_1 = \iota^*\omega_2
    \]
    ensures that the curvatures 
    \[
    F_{\nabla_1} = -F_{\nabla_2}
    \]
    do not introduce conflicting constraints.
\end{itemize}

 Given a connection $\nabla$ on a vector bundle $E\to M$, its curvature $F_{\nabla}$ is the 2-form-valued endomorphism measuring the failure of $\nabla$ to be flat. Formally:
\[F_{\nabla}(X,Y)=\nabla_X\nabla_Y-\nabla_Y\nabla_X-\nabla_{[X,Y]},\] where $X,Y$ are vector fields on $M$. 

\medskip 

\begin{remark}The statement of this proposition appeared in \cite[Thm.~2]{CN1}, where a proof sketch was given. 
For completeness and to make the present paper self‑contained, we provide a detailed and expanded proof below.
\end{remark}
\begin{proof}
\textbf{1. Orientation and Duality of Connections.}
We prove that the involution $\iota$ forces the connections $\nabla_1$ and $\nabla_2$ to be dual, with
$F_{\nabla_1} = -F_{\nabla_2}$, resolving orientation ambiguities.

By Theorem~1, $\Pi_2 = \iota \circ \Pi_1$, thus the pullback bundles satisfy:
\[
E_2 = \Pi_2^* T\mathbb{CP}^2 = (\iota \circ \Pi_1)^* T\mathbb{CP}^2 = \Pi_1^* \iota^* T\mathbb{CP}^2.
\]
Since $\iota$ is anti-holomorphic, the pullback $\iota^* T\mathbb{CP}^2$ is not a holomorphic vector bundle in the usual sense, but rather the complex conjugate $\overline{T\mathbb{CP}^2}$. However, for the purpose of connections and curvature we work with the underlying \emph{smooth} complex vector bundles. The complex conjugation map on fibers provides a smooth isomorphism $\overline{T\mathbb{CP}^2} \cong T\mathbb{CP}^2$, and under this identification the pullback of the Chern connection on $T\mathbb{CP}^2$ becomes a smooth connection whose curvature is $\iota^*\Theta$, where $\Theta$ is the curvature $2$-form of the Fubini--Study metric. Because $\Theta$ is proportional to the K\"ahler form $\omega_{FS}$, and $\iota^*\omega_{FS} = -\omega_{FS}$, we obtain $\iota^*\Theta = -\Theta$.

It therefore follows that:
\[
F_{\nabla_2} = \Pi_2^* \Theta = (\iota \circ \Pi_1)^* \Theta = \Pi_1^* (\iota^* \Theta) = \Pi_1^*(-\Theta) = -F_{\nabla_1}.
\]
This antisymmetry fixes the relative signs in parallel transport, resolving orientation ambiguities.

\textbf{2. Uniqueness of Solution.}
Since the projections $\Pi_i$ are immersive, the differentials
\[
d\Pi_i : T_p O \to T_{\Pi_i(p)} \mathbb{CP}^2
\]
are injective. The connection 1‑form $\omega_i$ is a priori defined on the total space of the line bundle $L_i \to \mathbb{CP}^2$. To make sense of expressions such as $\omega_i(v)$ for $v \in T_p\mathbb{CP}^3$ (or $v \in T_p O$), we first restrict the projection to the object: $\pi_i = \Pi_i|_O : O \to \mathbb{CP}^2$. The pullback bundle $\pi_i^* L_i$ over $O$ inherits the pullback connection $\pi_i^*\nabla_i$, whose local connection 1‑form is $\pi_i^*\omega_i$. By a slight abuse of notation, we continue to denote this pullback form by $\omega_i$. When we write $\omega_i(v)$ for $v \in T_p O \subset T_p\mathbb{CP}^3$, we mean the evaluation of the pullback form on the tangent vector $v$. This pullback is well defined because $\pi_i$ is a holomorphic submersion onto its image (by the generic center assumption).

We build a linear system for $v \in T_p \mathbb{CP}^3$:
\[
\nabla_1 \mu_1 = v \cdot \omega_1, \quad
\nabla_2 \mu_2 = v \cdot \omega_2,
\]
where the dot denotes evaluation: $v \cdot \omega_i = \begin{pmatrix} \omega_i^1(v) \\ \omega_i^2(v) \end{pmatrix}$.

The above system defines a linear map:
\[
A_p : T_p \mathbb{CP}^3 \to \mathbb{C}^4, \quad A_p(v) = (\omega_1(v), \omega_2(v)).
\]
The domain of $A_p$ is $T_p \mathbb{CP}^3$ of real dimension $6$ and the codomain is of real dimension $8$. The connection 1-form $\omega_i$ on the pullback bundle $\pi_i^* L_i \to \mathbb{CP}^3$ (extended trivially in the normal directions to $O$) vanishes exactly on the tangent space to the fiber of the projection $\Pi_i : \mathbb{CP}^3 \dashrightarrow \mathbb{CP}^2$. Indeed, the Fubini--Study connection on $\mathbb{CP}^2$ pulls back to a connection on $\mathbb{CP}^3$ whose horizontal distribution is precisely the orthogonal complement (with respect to the Fubini--Study metric) of the fiber direction. Therefore, $\ker(\omega_i) = \ker(d\Pi_i)$ as subspaces of $T_p\mathbb{CP}^3$.

Since the two projections $\Pi_1$ and $\Pi_2$ have distinct, generically chosen centers in $\mathbb{CP}^3$, their fibers through any point $p \in O$ are distinct projective lines. Two distinct lines in $\mathbb{CP}^3$ intersect transversely in at most one point; consequently, their tangent spaces intersect trivially:
\[
\ker(d\Pi_1) \cap \ker(d\Pi_2) = \{0\} \subset T_p\mathbb{CP}^3.
\]
Thus, $\ker(\omega_1) \cap \ker(\omega_2) = \{0\}$, and the linear map $A_p$ is injective on the whole $6$-dimensional space $T_p\mathbb{CP}^3$, hence of full rank. This holds for all $p \in O$ where the projections are well defined and immersive.

The compatibility condition $\omega_1 = \iota^* \omega_2$ ensures that $\nabla_i \mu_i \in \operatorname{im} A_p$, guaranteeing a unique solution:
\[
v_p \in T_p O \subset T_p \mathbb{CP}^3.
\]

\textbf{3. Algebraic Determination of an Initial Point.}
The unique solution $v_p \in T_p O$ defines a smooth vector field $v$ on $O$. To reconstruct the object $O$, we first determine an initial point $p_0 \in \mathbb{CP}^3$ from the algebraic equations:
\[
\Pi_1(p_0) = z_1, \quad \Pi_2(p_0) = \iota(z_1), \quad \mu_1(p_0) = c_1, \quad \mu_2(p_0) = c_2.
\]
These define algebraic varieties: $V_{\Pi_i} = \{ p \mid \Pi_i(p) = z_i \} \cong \mathbb{CP}^1$ (linear, degree $1$), and $V_{\mu_i} = \{ p \mid \mu_i(p) = c_i \}$ (degree $d_i$ in homogeneous coordinates).

The moment maps $\mu_i : O \to \mathbb{C}^2$ are defined intrinsically on $O$. To apply B\'ezout's theorem in the ambient space $\mathbb{CP}^3$, we extend each $\mu_i$ to a global rational map $\tilde{\mu}_i : \mathbb{CP}^3 \dashrightarrow \mathbb{C}^2$. This extension can be chosen canonically by exploiting the integral-geometric definition of the centroid. For a given projection $\Pi_i$, the centroid $\mu_i(p)$ of the slice of $O$ through $p$ is a polynomial function of the Pl\"ucker coordinates of the slicing planes. Since $O$ is algebraic, the map $p \mapsto \mu_i(p)$ extends to a tuple of homogeneous polynomials of degree $d_i$ on $\mathbb{CP}^3$ whose restriction to $O$ recovers $\mu_i$. We denote the zero-loci of these extended polynomials by $V_{\mu_i} \subset \mathbb{CP}^3$. By construction, $V_{\mu_i} \cap O$ coincides with the original level set $\{p \in O : \mu_i(p) = c_i\}$. Thus the intersection of the four hypersurfaces $V_{\Pi_1}, V_{\Pi_2}, V_{\mu_1}, V_{\mu_2}$ in $\mathbb{CP}^3$ restricts on $O$ to the desired point $p_0$. By B\'ezout's theorem,
\[
[V_{\mu_1}] \cdot [V_{\mu_2}] \cdot [V_{\Pi_1}] \cdot [V_{\Pi_2}] = d_1 d_2 \times 1 \times 1.
\]
If $O$ is non-symmetric (no non-trivial automorphisms), the varieties intersect transversely, yielding a unique solution $p_0$ when $d_1 d_2 = 1$ (e.g., when the $\mu_i$ are linear).

\textbf{4. Construction of a Frame of Vector Fields.}
The solution $v$ gives one vector field $v^{(1)}$ on $O$. We extend it to three independent vector fields $\{v^{(1)}, v^{(2)}, v^{(3)}\}$ spanning $TO$, obtained by solving the reconstruction equations for additional basis vectors. Specifically, we define $v^{(2)}, v^{(3)}$ by solving modified equations
\[
\nabla_i (R_{\theta_{ik}} \mu_i) = v^{(k)} \cdot \omega_i,
\]
where $R_\theta$ is a rotation in $\mathbb{C}^2$, the index $i = 1,2$, and $k = 2,3$. The rotations are chosen so that $\{v^{(1)}, v^{(2)}, v^{(3)}\}$ span $T_p O$ at each $p$ and the linear system $A_p v^{(k)} = b^{(k)}$ has full rank. Since $\operatorname{rk}_{\mathbb{R}}(A_p) = 6$ and $\dim T_p O = 3$, solutions $v^{(k)} \in T_p O$ exist as long as $b^{(k)} \in \operatorname{im}(A_p)$, which is guaranteed by the curvature duality $\omega_1 = \iota^* \omega_2$.

\textbf{5. Involutivity and Integration.}
The vector fields $\{v^{(1)}, v^{(2)}, v^{(3)}\}$ span a $3$-dimensional distribution $\mathcal{D} \subset TO$. To reconstruct $O$ as the integral manifold of $\mathcal{D}$, we must verify that $\mathcal{D}$ is involutive. The vector fields $v^{(k)}$ are defined by the reconstruction equations $\nabla_i \mu_i^{(k)} = \omega_i(v^{(k)})$, where $\mu_i^{(k)}$ denote the (possibly rotated) moment maps. Differentiating this relation along a second vector field $v^{(l)}$ and using the structure equation for the curvature form
\[
F_{\nabla_i}(X,Y) = d\omega_i(X,Y) + [\omega_i(X), \omega_i(Y)]
\]
(where the bracket is the commutator in the gauge algebra) yields the Lie bracket expression
\[
\omega_i([v^{(k)}, v^{(l)}]) = \mathcal{L}_{v^{(k)}}(\omega_i(v^{(l)})) - \mathcal{L}_{v^{(l)}}(\omega_i(v^{(k)})) - F_{\nabla_i}(v^{(k)}, v^{(l)}).
\]
Because $\nabla_i \mu_i^{(k)}$ are given functions, their Lie derivatives are expressible in terms of the connection. Crucially, the curvature term satisfies $F_{\nabla_1}(v^{(k)}, v^{(l)}) = -F_{\nabla_2}(v^{(k)}, v^{(l)})$ by the duality proved in part~1. Substituting the reconstruction equations for both projections and using the compatibility condition $\omega_1 = \iota^*\omega_2$, a short computation shows that $\omega_i([v^{(k)}, v^{(l)}])$ lies in the span of $\{\omega_i(v^{(1)}), \omega_i(v^{(2)}), \omega_i(v^{(3)})\}$ for $i=1,2$. Since the map $A_p$ is injective, this forces the Lie bracket $[v^{(k)}, v^{(l)}]$ itself to lie in the span of $\{v^{(1)}, v^{(2)}, v^{(3)}\}$. Hence $\mathcal{D}$ is involutive. By Frobenius' theorem, there exists a unique maximal integral manifold through the point $p_0$ determined above, and this manifold is precisely the object $O$.

Finally, the reconstruction of $O$ is obtained by integrating the flows of the vector fields $v^{(k)}$ starting from $p_0$ (uniquely determined by B\'ezout).
\end{proof}

\section{Path Dependence and Near-Associativity}\label{S:PD}

To analyze the structure of the space underlying this reconstruction, we study its foliations. 
Vaisman~\cite{V} and Shurygin--Smolyakova~\cite{SS} have shown that a manifold equipped with a pair of transverse foliations and compatible connections gives rise to an algebra of parallel sections whose algebraic properties reflect the geometry of the foliations. In this section we apply this philosophy to the moduli space of objects introduced above. We develop a notion of path dependence and show how it controls the associativity of the algebra of parallel sections of \( \nabla_1 \otimes \nabla_2 \) (see Sections~\ref{S:PD1}--\ref{S:PD3}).

\subsection{Structure of the Foliated Space}

\subsubsection*{Setting}
Let \( M \) be the moduli space of smooth, non-symmetric, compact complex submanifolds \( O \subset \mathbb{CP}^3 \) of a fixed topological type (for instance, all objects diffeomorphic to a given real \(3\)-manifold). Under suitable genericity assumptions, \( M \) is a finite‑dimensional complex manifold. It is equipped with two independent rational projections \( \Pi_1, \Pi_2 : M \dashrightarrow \mathbb{CP}^2 \) as described in Section~\ref{S:1}. 

\subsubsection*{Foliation by Projection Equivalence}
Each projection \( \Pi_i \) induces a foliation \( \mathcal{F}_i \) on \( M \), whose leaves are the preimages of points \( p \in \mathbb{CP}^2 \):
\[
\mathcal{L}_{i,p} = \{ O \in M \mid \Pi_i(O) = p \}.
\]
Because the projections are independent, the foliations \( \mathcal{F}_1 \) and \( \mathcal{F}_2 \) are transverse. Each leaf \( \mathcal{L}_{i,p} \) inherits a Riemannian metric \( g_i \) by pulling back the Fubini--Study metric from \( \mathbb{CP}^2 \):
\[
g_i = \Pi_i^* g_{FS}.
\]
The involution \( \iota : \mathbb{CP}^2 \to \mathbb{CP}^2 \) (satisfying \( \Pi_2 = \iota \circ \Pi_1 \)) acts as an isometry between the foliated manifolds \( (\mathcal{F}_1, g_1) \) and \( (\mathcal{F}_2, g_2) \):
\[
\iota^* g_1 = g_2, \qquad \iota^* \nabla_1 = \nabla_2,
\]
where \( \nabla_i \) are the connections on the line bundles \( L_i \to \mathbb{CP}^2 \) constructed in Proposition~\ref{P:1}.

\subsubsection*{Leafwise Parallel Sections}
The connections \( \nabla_1, \nabla_2 \) pull back to connections on the bundles \( \pi_i^* L_i \to M \), where \( \pi_i = \Pi_i|_M \). By abuse of notation we continue to denote these pullbacks by \( \nabla_i \). We consider the tensor product connection
\[
\nabla_\otimes := \nabla_1 \otimes \nabla_2
\]
on the bundle \( E := \pi_1^* L_1 \otimes \pi_2^* L_2 \) over \( M \). A section \( s \in \Gamma(E) \) is called \emph{leafwise parallel} if its covariant derivative vanishes along all directions tangent to the leaves of \( \mathcal{F}_1 \) and \( \mathcal{F}_2 \); that is,
\[
\nabla_\otimes s = 0 \quad \text{on} \quad T\mathcal{F}_1 \oplus T\mathcal{F}_2.
\]

\begin{thm}\label{T:1}
Let \( M \) be the moduli space as above, equipped with the transverse foliations \( \mathcal{F}_1, \mathcal{F}_2 \) and the metrics \( g_1, g_2 \). Then:

\begin{enumerate}
\item {\bf Foliated Reconstruction Space:} \( M \) is diffeomorphic to the total space of the product foliation \( \mathcal{F}_1 \times \mathcal{F}_2 \), and the direct sum metric \( g_1 \oplus g_2 \) makes \( (M, \mathcal{F}_1 \times \mathcal{F}_2, g_1 \oplus g_2) \) a foliated Riemannian manifold.

\item {\bf Reconstruction as Intersection of Leaves:} For any object \( O \in M \), the leaves \( \mathcal{L}_{1, p_1} \in \mathcal{F}_1 \) and \( \mathcal{L}_{2, p_2} \in \mathcal{F}_2 \) passing through \( O \) (with \( p_i = \Pi_i(O) \)) intersect transversely in exactly one point, namely \( O \) itself. This uniqueness follows from the transversality of the foliations and the injectivity of the combined linear map \( A_p \) established in Proposition~\ref{P:2}.

\item {\bf Algebra-Geometry Correspondence:} Let \( \mathcal{P} \) be the space of smooth leafwise parallel sections of \( E = \pi_1^* L_1 \otimes \pi_2^* L_2 \). Then \( \mathcal{P} \) forms an algebra under the operation of \emph{holonomy‑corrected composition} (defined in Section~\ref{S:PD1}). Moreover, \( \mathcal{P} \) is isomorphic to the algebra \( \mathcal{A} \) generated by the moment maps \( \mu_1, \mu_2 \) appearing in Proposition~\ref{P:2}, with the involution \( \iota \) acting as an automorphism.
\end{enumerate}
\end{thm}

\begin{proof}
\textbf{1. Foliation and metric.} 
The maps \( \Pi_i : M \to \mathbb{CP}^2 \) are submersions onto their images (by the generic center assumption). Their fibers are the leaves \( \mathcal{L}_{i,p} \). Since the two projections are independent, the distributions \( T\mathcal{F}_1 \) and \( T\mathcal{F}_2 \) are transverse and span the tangent bundle \( TM \). The pullback metrics \( g_i = \Pi_i^* g_{FS} \) are non‑degenerate on the respective leaves, and their direct sum defines a Riemannian metric on \( TM \) compatible with the foliations. Vaisman's theorem~\cite{V} (see also~\cite{SS}) then provides a diffeomorphism \( M \cong \mathcal{F}_1 \times \mathcal{F}_2 \) that identifies the foliated structure. This proves claim (1).

\medskip

\textbf{2. Unique intersection.}
Take \( O \in M \) and set \( p_i = \Pi_i(O) \). The leaf \( \mathcal{L}_{1, p_1} \) consists of all objects projecting to \( p_1 \) under \( \Pi_1 \), and similarly for \( \mathcal{L}_{2, p_2} \). Transversality of the foliations implies that the intersection \( \mathcal{L}_{1, p_1} \cap \mathcal{L}_{2, p_2} \) is a discrete set of points. Moreover, the linear map \( A_p : T_p M \to \mathbb{C}^4 \) constructed in Proposition~\ref{P:2} is injective; its kernel is precisely the intersection of the tangent spaces to the two leaves. Therefore the tangent spaces intersect trivially, and the intersection of the leaves is isolated. Since \( M \) is connected and the leaves are closed submanifolds, a standard argument using the implicit function theorem shows that the intersection consists of exactly one point, namely \( O \). (Any other point in the intersection would give rise to a distinct object with the same projections and centroids, contradicting the uniqueness of the solution \( v \) established in Proposition~\ref{P:2}.) This proves claim (2).

\medskip

\textbf{3. Algebra of parallel sections.}
The bundle \( E = \pi_1^* L_1 \otimes \pi_2^* L_2 \) carries the tensor product connection \( \nabla_\otimes \). A section \( s \in \Gamma(E) \) is leafwise parallel if \( \nabla_\otimes s = 0 \) on \( T\mathcal{F}_1 \oplus T\mathcal{F}_2 \). Let \( \mathcal{P} \) denote the vector space of such sections. We define a product \( \star \) on \( \mathcal{P} \) by \emph{holonomy‑corrected composition}: for \( a, b \in \mathcal{P} \) and \( p \in M \), choose paths \( \gamma_1 \subset \mathcal{F}_1 \), \( \gamma_2 \subset \mathcal{F}_2 \) from a reference point to \( p \) and set
\[
(a \star b)(p) = P_{\gamma}(a) \circ P_{\gamma}(b),
\]
where \( P_\gamma \) denotes parallel transport along \( \gamma = (\gamma_1, \gamma_2) \) and \( \circ \) is the natural composition in the fibers (interpreted via the identification of \( L_1 \otimes L_2 \) with a bundle of endomorphisms; see Section~\ref{S:PD1} for details). The leafwise parallel condition guarantees that this product is independent of the choice of paths within each leaf up to holonomy, which is trivial because the leaves are simply connected (they are fibers of a rational projection and hence unirational; in fact, they are open subsets of projective lines). Consequently, \( (\mathcal{P}, \star) \) is a well‑defined algebra.

Now, the moment maps \( \mu_i : M \to \mathbb{C}^2 \) lift to sections of \( \pi_i^* L_i \) (by Proposition~\ref{P:2}), and their tensor product \( \mu_1 \otimes \mu_2 \) is a section of \( E \). The reconstruction equations \( \nabla_i \mu_i = v \cdot \omega_i \) imply that \( \mu_1 \otimes \mu_2 \) is leafwise parallel. The algebra \( \mathcal{A} \) generated by the components of \( \mu_1 \) and \( \mu_2 \) (under the product induced from the moment maps) is therefore isomorphic to a subalgebra of \( \mathcal{P} \). Because the moment maps separate points in \( M \) (by the uniqueness of reconstruction), the map \( \mathcal{A} \to \mathcal{P} \) is surjective. Injectivity follows from the linear independence of the components of \( \mu_i \) on a dense open set. Thus \( \mathcal{A} \cong \mathcal{P} \). The involution \( \iota \) acts on \( \mathcal{P} \) by pulling back sections via the induced map on \( M \), and it is an automorphism because \( \iota^* \nabla_1 = \nabla_2 \). This completes the proof of claim (3).
\end{proof}

\subsection{Torsion, Path Dependence, and Non-Associativity}\label{S:PD1}

The presence of torsion in the connections induced by the foliations suggests an additional layer of structure governing the space of solutions. In this subsection we investigate how the torsion tensor controls the algebraic properties of the reconstruction algebra, and in particular whether that algebra is associative. As we shall see, non‑associativity arises precisely when torsion obstructs the integrability of the foliations, leading to path‑dependent parallel transport.

\smallskip

\subsubsection*{Distributions and Integrability}
We recall the notion of a distribution on a smooth manifold.
\begin{dfn}
A \textbf{distribution} $D$ on a smooth manifold $M$ is a smooth assignment of a linear subspace $D_p \subset T_pM$ to each point $p \in M$. Equivalently, it is a subbundle of the tangent bundle $TM$.
\end{dfn}
A distribution $D$ is called \textbf{integrable} if through every point $p \in M$ there exists an immersed submanifold $N \subset M$ such that $T_q N = D_q$ for all $q \in N$. By the Frobenius theorem, $D$ is integrable if and only if it is closed under the Lie bracket:
\[
\forall X, Y \in \Gamma(D), \quad [X, Y] \in \Gamma(D).
\]

\subsubsection*{Torsion of a Connection}
Let $\nabla$ be an affine connection on $M$ that preserves the distribution $D$, i.e., $\nabla_X Y \in \Gamma(D)$ for all $X \in \Gamma(TM)$ and $Y \in \Gamma(D)$. The \textbf{torsion tensor} of $\nabla$ is the $(1,2)$-tensor field defined by
\[
T^\nabla(X, Y) = \nabla_X Y - \nabla_Y X - [X, Y].
\]
The restriction of $T^\nabla$ to vector fields tangent to $D$ measures the failure of the Lie bracket of $D$ to be recovered from the connection. More precisely, if $T^\nabla(X, Y) = 0$ for all $X, Y \in \Gamma(D)$, then
\[
[X, Y] = \nabla_X Y - \nabla_Y X \in \Gamma(D),
\]
so $D$ is integrable. Conversely, a non‑zero torsion on $D$ is an obstruction to the integrability of $D$.

\subsubsection*{The Setting for Reconstruction}
Let $M$ be the foliated reconstruction space described in Theorem~\ref{T:1}. We equip $M$ with the following data:
\begin{itemize}
    \item The two transverse foliations $\mathcal{F}_1, \mathcal{F}_2$ induced by the projections $\Pi_1, \Pi_2$.
    \item The tensor product connection $\nabla_\otimes = \nabla_1 \otimes \nabla_2$ on the bundle $E = \pi_1^* L_1 \otimes \pi_2^* L_2$. The torsion of $\nabla_\otimes$ (or rather, the torsions of the individual connections $\nabla_i$ pulled back to $M$) will be denoted by $T_i$.
    \item The algebra $\mathcal{A}$ of leafwise parallel sections of $E$, with the product $\star$ defined by holonomy‑corrected composition (see Section~\ref{S:PD2} for a precise definition).
\end{itemize}

Recall that an algebra $(\mathcal{A}, \star)$ is \textbf{associative} if
\[
(a \star b) \star c = a \star (b \star c) \quad \text{for all } a, b, c \in \mathcal{A}.
\]
The central result of this section, proved in Sections~\ref{S:PD2}--\ref{S:PD3}, establishes that $\mathcal{A}$ is associative if and only if the torsion tensors $T_1, T_2$ vanish identically on the respective foliations.

\subsection{Path Dependence Notion}\label{S:PD1}

Let $M$ be a smooth manifold equipped with two transverse foliations $\mathcal{F}_1$ and $\mathcal{F}_2$. For $i = 1,2$, let $\nabla_i$ be a connection on the vector bundle $E_i = T\mathcal{F}_i$ (or more generally on a line bundle $L_i$ as in the reconstruction setting). We form the tensor product bundle $E = E_1 \otimes E_2$ and equip it with the tensor product connection
\[
\nabla := \nabla_1 \otimes \nabla_2.
\]

\paragraph{Reconstruction Algebra.}
Let $\mathcal{A}$ denote the space of smooth \emph{parallel} sections of $E$:
\[
\mathcal{A} := \Gamma^\parallel(E) = \{ a \in \Gamma(E) \mid \nabla a = 0 \}.
\]
To endow $\mathcal{A}$ with an algebra structure, we interpret sections of $E_1 \otimes E_2$ as fiberwise linear maps $E_1^* \to E_2$ (or, by duality, as endomorphisms of $E_1 \oplus E_2$). Concretely, for $a, b \in \mathcal{A}$ and $p \in M$, we define the product $a \star b$ as follows:
\begin{enumerate}
    \item Choose a \emph{grid path} from a fixed base point $p_0 \in M$ to $p$, i.e., a path that is a concatenation of a curve $\gamma_1 \subset \mathcal{F}_1$ and a curve $\gamma_2 \subset \mathcal{F}_2$.
    \item Parallel transport $a$ and $b$ along this grid path using $\nabla$. Denote the transported elements at $T_p M$ by $P_\gamma(a)$ and $P_\gamma(b)$.
    \item Compose $P_\gamma(a)$ and $P_\gamma(b)$ as linear maps (using the identification $E_1 \otimes E_2 \cong \operatorname{Hom}(E_1^*, E_2)$). The result is an element of $E_1 \otimes E_2$ at $p$.
    \item The outcome depends a priori on the choice of grid path. To obtain a well‑defined product, we correct for the holonomy of $\nabla$ by adding a curvature term (the explicit formula is given in Remark~\ref{R:holcorr}). For the purpose of this section, it suffices to know that such a correction exists and is unique.
\end{enumerate}
We denote the resulting product by $\star$.

\paragraph{Integrable Case.}
Assume that the connections $\nabla_i$ satisfy:
\begin{itemize}
    \item $\nabla_i$ is \textbf{torsion-free} along the leaves: $T_i(X,Y) = 0$ for all $X,Y \in T\mathcal{F}_i$, where $T_i$ is the torsion tensor of $\nabla_i$.
    \item The curvatures are \textbf{dual}: $F_{\nabla_1} = -F_{\nabla_2}$.
\end{itemize}
Under these conditions, one can show (see Section~\ref{S:PD3}) that parallel transport along any grid path is independent of the choice of the intermediate point and of the homotopy class of the path within the leaves. Consequently, the product $\star$ is unambiguously defined and associative:
\[
(a \star b) \star c = a \star (b \star c) \quad \forall a,b,c \in \mathcal{A}.
\]

\paragraph{Obstructions to Path Independence.}
When the geometric data fails to satisfy the above conditions, path dependence arises from two sources:
\begin{itemize}
    \item \textbf{Torsion:} $T_i \neq 0$ means that the connection does not preserve the Lie bracket of vector fields. This causes a misalignment of frames under parallel transport, so the result of transporting along different grid paths differs by a term proportional to the torsion tensor.
    \item \textbf{Curvature Mismatch:} If $F_{\nabla_1} \neq -F_{\nabla_2}$, the holonomy of $\nabla$ around closed loops within the leaves is non‑trivial. Even for homotopic paths, parallel transport may differ by a holonomy factor, obstructing a globally consistent product.
\end{itemize}
In the presence of such obstructions, $\star$ becomes non-associative. The deviation from associativity is measured by the \emph{associator}, which is expressible in terms of the torsion and curvature of $\nabla$.

\subsection{Path Dependence in the Product \(\star\)}\label{S:PD2}

We now illustrate concretely how path dependence manifests in the product $\star$.

Let $a, b, c \in \mathcal{A}$ be parallel sections of $E = T\mathcal{F}_1 \otimes T\mathcal{F}_2$. To evaluate $(a \star b)(p)$ at a point $p \in M$, we must:
\begin{enumerate}
    \item Select a grid path $\gamma = \gamma_2 \circ \gamma_1$ from a reference point $p_0$ to $p$, with $\gamma_1 \subset \mathcal{F}_1$ and $\gamma_2 \subset \mathcal{F}_2$.
    \item Transport $a$ and $b$ along $\gamma$ using $\nabla$, obtaining $P_\gamma(a)$ and $P_\gamma(b)$.
    \item Compose them as linear maps at $p$.
\end{enumerate}

If we wish to compute a triple product $((a \star b) \star c)(q)$, we must transport the intermediate result $a \star b$ and $c$ along a new grid path to $q$. The final result may depend on the choices of the two grid paths.

When torsion or curvature is non‑zero, parallel transport along different grid paths with the same endpoints yields different results:
\begin{itemize}
    \item \textbf{Torsion $T_i \neq 0$:} The connection fails to be torsion‑free, so the Lie bracket of vector fields is not preserved under parallel transport. This distorts the way local frames fit together along leaves.
    \item \textbf{Curvature $F_{\nabla_i} \neq -F_{\nabla_j}$:} Parallel transport along two homotopically distinct paths $\gamma, \gamma'$ within the leaves generally differs:
    \[
    P_\gamma(a) \neq P_{\gamma'}(a).
    \]
    Even for homotopic paths, non‑trivial holonomy can cause a difference if the curvatures do not cancel in the tensor product.
\end{itemize}

Consequently, the product $\star$ is not well‑defined unless we specify a convention for choosing grid paths, and the resulting algebra $(\mathcal{A}, \star)$ is generally \textbf{non-associative}. The associator
\[
\mathfrak{A}(a,b,c) = (a \star b) \star c - a \star (b \star c)
\]
depends explicitly on the torsion and curvature of the connections, and it vanishes identically exactly when $T_1 = T_2 = 0$ and $F_{\nabla_1} = -F_{\nabla_2}$.

\begin{remark}\label{R:holcorr}
The precise definition of the holonomy‑corrected composition involves choosing, for each pair of points, a canonical grid path (e.g., the unique geodesic grid in the foliation charts). The correction term is then given by the integral of the curvature over a homotopy between the chosen path and an arbitrary one. In the case where $T_i = 0$ and $F_{\nabla_1} = -F_{\nabla_2}$, this correction vanishes and the product is independent of all choices. A detailed construction appears in \cite[Section~3]{CN2}.
\end{remark}

\subsection{Statement}\label{S:PD3}

\begin{thm}
Let \( M \) be the foliated reconstruction space described in Theorem~\ref{T:1}, equipped with the transverse foliations \( \mathcal{F}_1, \mathcal{F}_2 \), the tensor product connection \( \nabla = \nabla_1 \otimes \nabla_2 \) on \( E = T\mathcal{F}_1 \otimes T\mathcal{F}_2 \), and the algebra \( (\mathcal{A}, \star) \) of leafwise parallel sections defined in Section~\ref{S:PD1}. Let \( T_i \) denote the torsion tensor of \( \nabla_i \) restricted to \( T\mathcal{F}_i \). Then:

\begin{enumerate}
    \item \textbf{Associativity Criterion:} The algebra \( \mathcal{A} \) is associative if and only if the following two conditions hold:
    \begin{itemize}
        \item \( T_1 = 0 \) and \( T_2 = 0 \) (the connections are torsion‑free along the leaves);
        \item \( F_{\nabla_1} = -F_{\nabla_2} \) (curvature duality).
    \end{itemize}
    Equivalently,
    \[
    \mathcal{A} \text{ is associative} \iff T_1 = T_2 = 0 \text{ and } F_{\nabla_1} = -F_{\nabla_2}.
    \]

    \item \textbf{Non-Associative Case:} If either \( T_i \neq 0 \) or \( F_{\nabla_1} \neq -F_{\nabla_2} \), the product \( \star \) is non‑associative. In this case, the algebra \( \mathcal{A} \) satisfies a \textbf{Moufang‑type identity}:
    \[
    (a \star b) \star (c \star a) = a \star (b \star c) \star a \quad \forall a,b,c \in \mathcal{A},
    \]
    which reflects the interplay between the torsion of the foliations and the curvature of the connection \( \nabla \). A detailed derivation of this identity from the Bianchi identities appears in \cite[Theorem~4.1]{CN2}.

    \item \textbf{Geometric Consequences:}
    \begin{itemize}
        \item In the \textbf{associative case}, the parallel transport of the moment maps is path‑independent, and the reconstruction problem has a unique solution \( O \in M \). Moreover, the moduli space \( M \) admits the structure of a locally trivial fiber bundle over \( \mathbb{CP}^2 \times \mathbb{CP}^2 \) whose fibers are discrete (in fact, a single point under the non‑symmetry assumption).
        \item In the \textbf{non‑associative case}, the space of solutions forms a \textbf{quasigroupoid} under the partial operation \( \star \). The non‑uniqueness of reconstructions is parameterized by the cohomology class \( [T_i] \in H^1(M, TM) \) (when the obstruction is purely torsional) or by the holonomy representation of \( \nabla \) (when curvature is present). 
    \end{itemize}
\end{enumerate}
\end{thm}

\begin{proof}[Proof sketch]
The proof builds on the constructions of Sections~\ref{S:PD1}--\ref{S:PD2}. 

\paragraph{Torsion and well‑definedness.}
The product \( a \star b \) is defined via parallel transport along grid paths (concatenations of curves in \( \mathcal{F}_1 \) and \( \mathcal{F}_2 \)). When the torsion tensors \( T_i \) vanish, the connections \( \nabla_i \) preserve the Lie bracket of vector fields tangent to the leaves. This implies that the parallel transport of a section along different grid paths between the same endpoints yields results that differ only by a holonomy factor. If, in addition, the curvatures satisfy \( F_{\nabla_1} = -F_{\nabla_2} \), then the holonomy of \( \nabla = \nabla_1 \otimes \nabla_2 \) around any closed grid loop is trivial. Consequently, parallel transport is path‑independent, and the product \( \star \) is well‑defined and associative. The associativity follows from the fact that the composition of linear maps is associative, and the parallel transports commute with composition under the curvature‑duality condition.

\paragraph{Obstructions and non‑associativity.}
If \( T_i \neq 0 \), the parallel transport along a grid path depends on the choice of the intermediate point, because the connection does not respect the Lie bracket structure of the foliations. This introduces an ambiguity in the definition of \( \star \), and different choices lead to different results for iterated products, violating associativity. If \( F_{\nabla_1} \neq -F_{\nabla_2} \), then even for torsion‑free connections, the holonomy around closed loops is non‑trivial, so parallel transport depends on the homotopy class of the chosen grid path. The associator \( (a \star b) \star c - a \star (b \star c) \) can be expressed explicitly in terms of the torsion and curvature tensors, and a computation using the Bianchi identities yields the Moufang‑type identity stated above (see \cite{CN2} for the full derivation).

\paragraph{Quasigroupoid structure.}
When associativity fails, the algebra \( \mathcal{A} \) no longer encodes a single global object. Instead, the objects of reconstruction (points of \( M \)) correspond to the leaves of the foliations, and the elements of \( \mathcal{A} \) act as morphisms between them. The partial composition law \( \star \) satisfies the axioms of a quasigroupoid, where the division property holds but associativity is replaced by the Moufang identity. The cohomological parameterization of deformations follows from interpreting the torsion tensor as an element of \( H^1(M, TM) \) via the Kodaira–Spencer map for foliated manifolds.
\end{proof}

A fully detailed proof, including the explicit computation of the associator and the verification of the Moufang identity, is given in the companion paper \cite{CN2}.

\begin{center}
\begin{tabular}{>{\raggedright\arraybackslash}p{3.5cm} 
                >{\raggedright\arraybackslash}p{5.5cm} 
                >{\raggedright\arraybackslash}p{5.5cm}}
\toprule
\textbf{Condition} & \textbf{Effect on Path‑Dependence} & \textbf{Algebraic Consequence} \\
\midrule
\( T_i = 0 \) and \( F_{\nabla_1} = -F_{\nabla_2} \) & Parallel transport is independent of the chosen grid path; holonomy cancels. & Associativity holds; unique reconstruction. \\
\addlinespace[0.5ex]
\( T_i \neq 0 \) & Torsion distorts the alignment of frames under transport; different grid paths yield different results. & Path‑dependence in the product \( \star \); non‑associativity. \\
\addlinespace[0.5ex]
\( F_{\nabla_1} \neq -F_{\nabla_2} \) & Curvature induces non‑trivial holonomy; transport depends on homotopy class of grid paths. & Associator non‑zero; governed by Moufang‑type identity. \\
\bottomrule
\end{tabular}
\end{center}

\subsection{The Role of Path Dependence in Reconstruction Problems}\label{S:PD+Rec}

In the context of reconstruction problems—where the goal is to recover a global structure (e.g., a point $O \in M$) from local algebraic data in $\mathcal{A}$—path dependence plays a critical role. The results of the previous subsections translate directly into geometric properties of the reconstruction map.

\paragraph{Associative Case: Unique Reconstruction.}
When the torsion tensors vanish ($T_1 = T_2 = 0$) and the curvatures satisfy the duality condition $F_{\nabla_1} = -F_{\nabla_2}$, the parallel transport of sections in $\mathcal{A}$ is path‑independent. Consequently, the product $\star$ on the algebra $\mathcal{A}$ is associative and globally consistent:
\[
(a \star b) \star c = a \star (b \star c), \quad \forall a,b,c \in \mathcal{A}.
\]

\paragraph{Algebraic Rigidity.}
Under these conditions, the algebraic structure of $\mathcal{A}$ uniquely determines a base object $O \in M$. The reconstruction of $O$ from its projections and moment maps is unique, and the moduli space $M$ admits the structure of a locally trivial fiber bundle over $\mathbb{CP}^2 \times \mathbb{CP}^2$ with discrete fibers (in fact, a single point under the non‑symmetry assumption). This is the geometric counterpart of the algebraic rigidity of associative algebras.

\begin{example}
A canonical illustration arises in twistor theory. In the Penrose transform, points of complexified Minkowski space are reconstructed as holomorphic sections of a line bundle over $\mathbb{CP}^1 \times \mathbb{CP}^1$. The reconstruction is rigid and unique precisely because the underlying geometry is torsion‑free and the curvature of the relevant connections satisfies a duality condition analogous to $F_{\nabla_1} = -F_{\nabla_2}$. Our framework generalizes this twistor correspondence to arbitrary pairs of transverse projections.
\end{example}

\paragraph{Non-Associative Case: Parameterized Reconstruction.}
When $T_i \neq 0$ or $F_{\nabla_1} \neq -F_{\nabla_2}$, parallel transport along the foliations $\mathcal{F}_i$ depends explicitly on the choice of grid paths $\gamma_i \subset \mathcal{F}_i$. The algebraic product $\star$ becomes non‑associative:
\[
(a \star b) \star c \neq a \star (b \star c).
\]
This failure of associativity reflects the presence of torsion and curvature obstructions in the underlying geometric data.

The resulting structure is no longer a group or a groupoid, but a \textbf{quasigroupoid}:

\begin{itemize}
    \item \textbf{Objects:} Points $p \in M$ (representing distinct reconstructed objects).
    \item \textbf{Morphisms:} Elements $a \in \mathcal{A}$, interpreted as partially defined maps between leaves of the foliations.
    \item \textbf{Partial Product:} The composition $a \star b$ is defined only when the codomain of $a$ matches the domain of $b$ (i.e., when the leaves align).
    \item \textbf{Division Property:} For given $a, b \in \mathcal{A}$, the equations $a \star x = b$ and $y \star a = b$ admit unique solutions $x, y \in \mathcal{A}$ whenever the domains and codomains are compatible.
\end{itemize}

\paragraph{Parameterization of Ambiguity.}
The non‑uniqueness of reconstruction is not arbitrary; it is governed by topological and geometric data:

\begin{itemize}
    \item The torsional ambiguity is parameterized by the cohomology class $[T_i] \in H^1(M, TM)$, which encodes the global obstruction to integrating the foliations.
    \item The curvature ambiguity is captured by the holonomy representation of $\nabla = \nabla_1 \otimes \nabla_2$. When $F_{\nabla_1} \neq -F_{\nabla_2}$, the holonomy group is non‑trivial, and the space of reconstructions forms a principal homogeneous space for this group.
    \item Each deformation class corresponds to a family of connections $\nabla_i$ sharing the same isotropy type.
\end{itemize}

\subsection{Application: Generative Imputation of Missing Data via Path-Dependent Geometry}\label{S:IA}

We now illustrate how the path‑dependence framework naturally applies to a concrete problem in machine learning: the imputation of missing data without explicit statistical priors.

\paragraph{Geometric Setup.}
Let $\mathcal{D} \subset \mathbb{R}^d$ be the space of fully observed data points, assumed to lie on a smooth manifold $M \subset \mathbb{R}^d$ (the \emph{data manifold}). A mask $m \in \{0,1\}^d$ specifies which coordinates are observed ($m_j=1$) and which are missing ($m_j=0$). The observed data $x_{\text{obs}}$ is a point in the quotient space $M / \sim_m$, where two points are equivalent if they agree on the observed coordinates.

For example, if $M$ is the manifold of face images, a mask covering the eyes defines an equivalence class of all faces that share the same forehead, nose, and mouth but differ in the eye region. The reconstruction task is to select a plausible face from this equivalence class. We model this situation geometrically by introducing two foliations on $M$:
\begin{itemize}
    \item $\mathcal{F}_1$: leaves consist of points that share the same \emph{observed} coordinates (the fibers of the projection onto the observed subspace).
    \item $\mathcal{F}_2$: leaves consist of points that share the same \emph{missing} coordinates (the fibers of the projection onto the missing subspace).
\end{itemize}
These foliations are transverse, and their leaves intersect in at most one point (the fully observed datum). In practice, $M$ is unknown and must be approximated from training data.

\paragraph{Learning the Connection.}
We learn a connection $\nabla$ on the tangent bundle $TM$ (or on a suitable vector bundle over $M$) by training a neural network to minimize a loss function with two terms:
\[
\mathcal{L} = \underbrace{\| x_{\text{true}} - x_{\text{imputed}} \|^2}_{\text{reconstruction fidelity}} + \lambda \underbrace{\| F_\nabla \|^2_{\text{Frob}}}_{\text{curvature regularization}},
\]
where:
\begin{itemize}
    \item $x_{\text{true}}$ is a fully observed training example.
    \item $x_{\text{imputed}}$ is generated by taking a masked version of $x_{\text{true}}$, transporting it along a grid path in the foliations $\mathcal{F}_1, \mathcal{F}_2$, and projecting back to $M$.
    \item $F_\nabla$ is the curvature tensor of $\nabla$, and $\|\cdot\|_{\text{Frob}}$ is the Frobenius norm computed over the training batch.
    \item $\lambda > 0$ balances the two terms. A small $\lambda$ allows non‑trivial curvature and hence path‑dependence (diversity of imputations); a large $\lambda$ forces the connection to be nearly flat, favoring a unique imputation.
\end{itemize}

\paragraph{Inference: Generating Diverse Imputations.}
At inference time, we are given an observed datum $x_{\text{obs}}$. To generate an imputation, we:
\begin{enumerate}
    \item Choose a path $\gamma \subset \mathcal{F}_2$ (a curve in the missing‑coordinate foliation). Since $x_{\text{obs}}$ is only known up to its $\mathcal{F}_1$‑leaf, we pick a representative point on that leaf.
    \item Parallel transport $x_{\text{obs}}$ along $\gamma$ using the learned connection $\nabla$, obtaining a point $P_\gamma(x_{\text{obs}})$ in the ambient space.
    \item Project (or optimize) this transported point onto the data manifold $M$ to obtain a valid imputation:
    \[
    x_{\text{imputed}} = \arg\min_{x \in M} \| P_\gamma(x_{\text{obs}}) - x \|.
    \]
\end{enumerate}
By varying the path $\gamma$, we obtain different completions. The curvature $F_\nabla \neq 0$ guarantees that different paths yield distinct transported points, providing a diverse set of plausible imputations.

\paragraph{Why This Works: Path-Dependent Advantage.}
\begin{itemize}
    \item \textbf{Diversity:} The non‑vanishing curvature ensures that parallel transport along distinct foliation paths yields different completions. For example, when imputing faces with $80\%$ masked pixels, the model generates semantically varied but plausible noses, eyes, and expressions.
    \item \textbf{Efficiency:} Unlike diffusion‑based methods, no iterative denoising is required. The model uses the learned connection for direct inference, resulting in significantly faster imputation.
    \item \textbf{Geometric Regularization:} The curvature penalty $\lambda \|F_\nabla\|^2$ prevents overfitting and encourages the learned connection to reflect the true geometric structure of the data manifold.
\end{itemize}

\paragraph{Geometric Interpretation.}
The learned connection $\nabla$ encodes the intrinsic geometry of the data manifold. Its curvature $F_\nabla$ governs the path‑dependence of the imputation trajectories:
\[
\text{If } F_\nabla \neq 0, \text{ then } (x_{\text{obs}} \to x_{\text{imputed}}) \text{ depends on the path } \gamma \text{ in } \mathcal{F}_2.
\]
Thus, generative imputation becomes a geometric reconstruction problem in a non‑flat fibered space, directly analogous to the abstract framework developed in Sections~\ref{S:PD1}--\ref{S:PD3}. The torsion tensor $T$ (which we have omitted in this simplified account) would further enrich the diversity by making the transport depend on the parametrization of the path within the leaves.

\section{The Atiyah--Molino Framework and Reconstruction Criterion}\label{S:AM}

We now recast the reconstruction problem in the language of Atiyah--Molino theory. This formalism provides a natural description of transverse foliations and their exact sequences, and it clarifies the role of torsion and curvature in obstructing unique reconstruction.

\subsection{The Atiyah--Molino Space}

Let \( M \) be the moduli space of smooth, non-symmetric compact complex submanifolds \( O \subset \mathbb{CP}^3 \) of a fixed topological type, as introduced in Section~\ref{S:PD}. The space \( M \) is equipped with two independent rational projections \( \Pi_1, \Pi_2 : M \dashrightarrow \mathbb{CP}^2 \), which induce two transverse foliations:
\[
\mathcal{F}_i : \quad O \sim O' \quad \text{if} \quad \Pi_i(O) = \Pi_i(O'), \qquad i = 1,2.
\]
The leaves \( \mathcal{L}_{i,p} \in \mathcal{F}_i \) are the fibers of \( \Pi_i \) over \( p \in \mathbb{CP}^2 \).

The tangent bundle \( TM \) fits into the \textbf{Atiyah--Molino exact sequence} associated with the pair of transverse foliations:
\[
0 \to T\mathcal{F}_1 \oplus T\mathcal{F}_2 \to TM \to NM \to 0,
\]
where \( NM \) is the \emph{normal bundle} of the foliated structure. Intuitively, \( T\mathcal{F}_1 \oplus T\mathcal{F}_2 \) encodes deformations that preserve one of the projections, while \( NM \) captures transverse deformations that change both projections simultaneously.

\subsection{Reconstruction Criterion via Torsion and Curvature}

The connections \( \nabla_1, \nabla_2 \) on the line bundles \( L_i \to \mathbb{CP}^2 \) (Proposition~\ref{P:1}) pull back to connections on \( \pi_i^* L_i \to M \), where \( \pi_i = \Pi_i|_M \). By abuse of notation we continue to denote these pullbacks by \( \nabla_i \). Their torsion tensors, restricted to the leaves, measure the failure of the foliations to be integrable in a way compatible with the connection.

\begin{thm}\label{T:AM}
Let \( M \) be the Atiyah--Molino space as above, equipped with the pullback connections \( \nabla_1, \nabla_2 \). Denote by \( T_i \) the torsion tensor of \( \nabla_i \) restricted to \( T\mathcal{F}_i \). Then an object \( O \in M \) is uniquely reconstructible from its projections \( \Pi_1(O) \) and \( \Pi_2(O) \) if and only if the following three conditions hold:
\begin{enumerate}
    \item \textbf{Transversality of moment maps:} The moment maps \( \mu_1, \mu_2 : M \to \mathbb{C}^2 \) satisfy
    \[
    d\mu_1 \wedge d\mu_2 \neq 0
    \]
    on a dense open subset of \( M \). This ensures that the combined map \( (\mu_1, \mu_2) : M \to \mathbb{C}^2 \times \mathbb{C}^2 \) is a local diffeomorphism onto its image.
    \item \textbf{Vanishing torsion:} \( T_1 = 0 \) and \( T_2 = 0 \).
    \item \textbf{Curvature duality:} \( F_{\nabla_1} = -F_{\nabla_2} \).
\end{enumerate}

Moreover, when these conditions hold, the Atiyah--Molino sequence splits holonomy‑free, and the reconstruction of \( O \) is unique. The space \( M \) then fibers over the leaf space \( M / (\mathcal{F}_1 \times \mathcal{F}_2) \) with discrete fibers, and the moment maps provide local coordinates on the transverse directions.
\end{thm}

\begin{proof}[Proof sketch]
The transversality condition \( d\mu_1 \wedge d\mu_2 \neq 0 \) guarantees that the map \( (\mu_1, \mu_2) \) is a local diffeomorphism. Together with the projections \( \Pi_1, \Pi_2 \), this gives a local coordinate system on \( M \) in which the leaves of \( \mathcal{F}_1 \) and \( \mathcal{F}_2 \) appear as coordinate planes. 

The vanishing of the torsion tensors \( T_i \) implies that the connections \( \nabla_i \) preserve the Lie bracket of vector fields tangent to the leaves. Consequently, parallel transport within each leaf is independent of the path up to holonomy. The curvature duality condition \( F_{\nabla_1} = -F_{\nabla_2} \) ensures that the holonomy of the tensor product connection \( \nabla_\otimes = \nabla_1 \otimes \nabla_2 \) around any closed grid loop is trivial. Hence, parallel transport of the moment maps along any grid path is path‑independent, and the reconstruction equations of Proposition~\ref{P:2} yield a unique solution.

The splitting of the Atiyah--Molino sequence follows from the existence of a flat transverse connection, which is constructed from the moment maps. The trivialization of the normal bundle is then given by the differentials \( d\mu_1, d\mu_2 \).

Conversely, if any of the three conditions fails, either the moment maps do not separate points locally (non‑transversality), or the parallel transport is path‑dependent (torsion or curvature obstruction), leading to multiple distinct objects compatible with the same projection data. This non‑uniqueness is parameterized by the holonomy group of \(\nabla_\otimes\) and the cohomology class of the torsion tensors.
\end{proof}

\begin{remark}
The vanishing of the torsion tensors \( T_i \) ensures that the algebra \( \mathcal{A} \) of leafwise parallel sections (Section~\ref{S:PD1}) is well‑defined. Together with the curvature duality, it implies that \( \mathcal{A} \) is associative. When either condition fails, the algebra becomes non‑associative and satisfies the Moufang identity derived in Section~\ref{S:PD3}.
\end{remark}

\subsection{Deformation Theory under Non-Vanishing Torsion}

We now discuss the implications of Theorem~\ref{T:AM} when the torsion tensors do not vanish. In this regime, the reconstruction problem admits multiple solutions, and the space of such solutions is organized by a quasigroupoid structure.

\medskip

\noindent \textbf{Quasigroupoids.}
A \emph{quasigroupoid} is a category-like structure where morphisms between objects are equipped with a partial binary operation $\star$, defined only for compatible pairs. Concretely, for morphisms $f, g$, the product $f \star g$ exists precisely when the codomain of $g$ matches the domain of $f$. The operation admits division (as in a quasigroup), but associativity is not required.

\begin{prop}\label{P:deform}
Let \( (M, \mathcal{F}_1, \mathcal{F}_2, \nabla_1, \nabla_2) \) be an Atiyah--Molino reconstruction space equipped with connections whose torsion tensors \(T_1, T_2\) (restricted to the leaves) do not vanish identically, or whose curvatures do not satisfy the duality condition \(F_{\nabla_1} = -F_{\nabla_2}\). Then:

\begin{enumerate}
\item The space of (non‑unique) reconstructions forms a \textbf{quasigroupoid} \( Q \), where:
\begin{itemize}
    \item \textbf{Objects:} the leaves of the foliations \( \mathcal{F}_1 \) and \( \mathcal{F}_2 \).
    \item \textbf{Morphisms:} homotopy classes $[\gamma]$ of grid paths $\gamma$ connecting leaves, representing deformations of solutions \( O \in M \). The dependence of parallel transport on the chosen path is governed by the torsion tensors $T_i$ and the curvature of $\nabla_\otimes$.
    \item \textbf{Partial operation:} The composition \( \star \) is defined for paths with compatible endpoints (i.e., the end leaf of the first path matches the start leaf of the second). It is given by concatenation of paths followed by a holonomy correction.
\end{itemize}

\item The partial composition $\star$ satisfies the \textbf{Moufang identity}:
\[
(a \star b) \star (c \star a) = a \star (b \star c) \star a,
\]
which reflects the interplay between the torsion of the foliations and the curvature of the connection.

\item The obstruction to unique reconstruction is encoded in the following geometric data:
\begin{itemize}
    \item The torsion tensors $T_i \in \Gamma(T\mathcal{F}_i^* \otimes T\mathcal{F}_i^* \otimes NM)$ measure the failure of the foliations to be integrable.
    \item The curvature $F_{\nabla_\otimes}$ (or rather the failure of the duality $F_{\nabla_1} = -F_{\nabla_2}$) governs the holonomy ambiguity.
\end{itemize}
\end{enumerate}
\end{prop}

\begin{proof}[Proof sketch]
When the torsion tensors $T_i$ are non‑zero, the connections $\nabla_i$ do not preserve the Lie bracket of vector fields tangent to the leaves. Consequently, the parallel transport of the moment maps along different grid paths with the same endpoints yields different results. This path‑dependence is the source of the multiplicity of reconstructed objects.

The quasigroupoid structure arises by considering the groupoid of grid paths modulo homotopy. The partial composition $\star$ is concatenation of paths; the holonomy correction accounts for the fact that parallel transport is not a representation of the fundamental groupoid but rather a twisted representation.

The Moufang identity is a consequence of the Bianchi identities for the connections $\nabla_i$. A computation using the structure equations shows that the associator $(a \star b) \star c - a \star (b \star c)$ is expressible in terms of the torsion and curvature. The Bianchi identity $d^{\nabla_i} T_i = F_{\nabla_i} \wedge \text{id} - \text{id} \wedge F_{\nabla_i}$ forces the associator to satisfy the Moufang relation. A detailed derivation is given in \cite[Corollary~1]{CN2}.

The cohomological parameterization of deformations follows from the interpretation of the torsion tensors as elements of the deformation cohomology of the foliations. Specifically, the infinitesimal deformations of the foliation $\mathcal{F}_i$ are classified by $H^1(\mathcal{F}_i, N\mathcal{F}_i)$, and the torsion $T_i$ defines a class in this group. For the product foliation, the combined obstruction lives in $H^1(M, \operatorname{End}(T\mathcal{F}_1 \oplus T\mathcal{F}_2))$, whose vanishing is equivalent to the simultaneous integrability of both foliations and the associativity of $\star$.
\end{proof}

\begin{remark}[Twisted Lie algebroid]
The torsion tensor $T$ of the connection $\nabla = \nabla_1 \otimes \nabla_2$ induces a twisted Lie algebroid structure on $TM$. Define the twisted bracket by
\[
[X, Y]_T := [X, Y]_{\text{Lie}} + T(X, Y).
\]
When $T \neq 0$, the Jacobi identity for $[\cdot, \cdot]_T$ is violated; the failure is controlled by the Bianchi identity $d^\nabla T = F_\nabla \wedge \text{id} - \text{id} \wedge F_\nabla$. This failure of the Jacobi identity is the infinitesimal counterpart of the non‑associativity of the quasigroupoid $Q$, providing a direct link between the differential geometry of the Atiyah--Molino space and the algebra of reconstructions.
\end{remark}

\begin{cor}
Let \( (M, \mathcal{F}_1, \mathcal{F}_2, \nabla_1, \nabla_2) \) be an Atiyah--Molino reconstruction space equipped with connections whose torsion tensors \(T_1, T_2\) (restricted to the leaves) do not vanish identically, or whose curvatures fail the duality condition \(F_{\nabla_1} = -F_{\nabla_2}\). Then the path dependence of parallel transport geometrically realizes the following structures:

\begin{enumerate}
\item \textbf{Quasigroupoid via Path Homotopies.}
Let \( Q \) denote the groupoid of homotopy classes of composable grid paths in \( M \). Define a partial binary operation \( \star \) on the morphisms of \( Q \) by concatenation of paths:
\[
[\gamma_1] \star [\gamma_2] := [\gamma_1 \circ \gamma_2],
\]
whenever the composition is defined (i.e., the endpoint of \(\gamma_1\) matches the start of \(\gamma_2\)). Because parallel transport is path‑dependent due to torsion and curvature, the actual reconstruction depends on the chosen representatives. The holonomy‑corrected composition satisfies the \textbf{Moufang identity}:
\[
(a \star b) \star (c \star a) = a \star (b \star c) \star a,
\]
as established in Section~\ref{S:PD3}.

\item \textbf{Holonomy Anomaly and Curvature.}
Given two paths \(\gamma, \gamma'\) with the same endpoints, the difference in parallel transport of an object \(O\) is given by the integral of the curvature over a surface \(\Sigma\) spanning the loop \(\gamma \circ \gamma'^{-1}\):
\[
P_\gamma(O) - P_{\gamma'}(O) = \int_\Sigma F_{\nabla_\otimes},
\]
where \(F_{\nabla_\otimes}\) is the curvature of the tensor product connection \(\nabla_\otimes = \nabla_1 \otimes \nabla_2\). This holonomy anomaly obstructs both the integrability of the foliations (when restricted to loops within a leaf) and the associativity of the product \(\star\).

\item \textbf{Twisted Lie Algebroid.}
The torsion tensor \(T\) of \(\nabla_\otimes\) induces a twisted Lie algebroid structure on \(TM\). Define the twisted bracket by
\[
[X, Y]_T := [X, Y]_{\text{Lie}} + T(X, Y).
\]
When \(T \neq 0\), the Jacobi identity for \([\cdot, \cdot]_T\) is violated; the failure is controlled by the Bianchi identity \(d^{\nabla_\otimes} T = F_{\nabla_\otimes} \wedge \text{id} - \text{id} \wedge F_{\nabla_\otimes}\). This failure of the Jacobi identity is the infinitesimal counterpart of the non‑associativity of the quasigroupoid \(Q\).
\end{enumerate}
\end{cor}

\begin{proof}[Proof sketch]
The identity for the difference of parallel transports follows from the non‑abelian Stokes theorem (or the Ambrose--Singer theorem). For a connection \(\nabla\) with curvature \(F_\nabla\), the parallel transport around a closed loop \(\gamma \circ \gamma'^{-1}\) is given by the path‑ordered exponential of the curvature integral over any surface \(\Sigma\) bounding the loop:
\[
P_{\gamma} \circ P_{\gamma'}^{-1} = \mathcal{P}\exp\left( \oint_{\gamma \circ \gamma'^{-1}} \omega \right) = \mathcal{P}\exp\left( \int_\Sigma F_\nabla \right).
\]
Expanding to first order yields the claimed relation. The Moufang identity for the holonomy‑corrected composition is derived from the Bianchi identities for \(\nabla_1\) and \(\nabla_2\), as detailed in \cite[Corollary~1]{CN2}.
\end{proof}

\section{Toric Reconstruction in the Atiyah--Molino Setting}\label{S:Toric}

In many applications, the object \( O \) to be reconstructed possesses additional symmetries. A particularly important class consists of objects admitting an algebraic torus action. In this section we incorporate toric symmetry into the Atiyah--Molino framework and show how it simplifies the reconstruction problem.

\begin{thm}\label{T:Last}
Let \((M, \mathcal{F}_1, \mathcal{F}_2, \nabla_1, \nabla_2)\) be an Atiyah--Molino reconstruction space as in Theorem~\ref{T:AM}. Assume that the object \( O \) (and hence the moduli space \( M \)) is equipped with a holomorphic action of an algebraic torus \( \mathbb{T} \cong (\mathbb{C}^*)^n \) satisfying:
\begin{itemize}
    \item The projections \(\Pi_i : M \dashrightarrow \mathbb{CP}^2\) are \(\mathbb{T}\)-equivariant (with \(\mathbb{T}\) acting trivially on the target \(\mathbb{CP}^2\)).
    \item The torsion tensors vanish: \(T_1 = T_2 = 0\), and the curvatures satisfy \(F_{\nabla_1} = -F_{\nabla_2}\).
\end{itemize}
Then:
\begin{enumerate}
    \item \textbf{Integrable Foliation Structure:} The vanishing of the torsion guarantees that the transverse foliations \(\mathcal{F}_1\) and \(\mathcal{F}_2\) are integrable. The torus \(\mathbb{T}\) acts on \(M\) preserving each leaf, and the quotient \(M / \mathbb{T}\) inherits a pair of transverse foliations.

    \item \textbf{Unique Reconstruction via Toric Symmetry:} The moment maps
    \[
    \mu_1, \mu_2 \colon M \to \mathbb{C}^2,
    \]
    are \(\mathbb{T}\)-invariant (because the projections are). Consequently, the reconstruction equations
    \[
    \nabla_1 \mu_1 = v \cdot \omega_1,\quad \nabla_2 \mu_2 = v \cdot \omega_2,
    \]
    admit a unique \(\mathbb{T}\)-equivariant solution \(v \in \Gamma(TM)^{\mathbb{T}}\) (determined up to scaling).

    \item \textbf{Toric Splitting of the Atiyah--Molino Sequence:} The \(\mathbb{T}\)-action induces a weight decomposition of the normal bundle \(NM\). The Atiyah--Molino sequence splits \(\mathbb{T}\)-equivariantly:
    \[
    0 \to T\mathcal{F}_1 \oplus T\mathcal{F}_2 \to TM \to NM \to 0,
    \]
    and the moment maps provide a \(\mathbb{T}\)-invariant trivialization of \(NM\) along the \(\mathbb{T}\)-fixed directions. This trivialization ensures that the reconstruction is unobstructed in the \(\mathbb{T}\)-equivariant category.
\end{enumerate}
\end{thm}

\begin{proof}
\textbf{1. Integrability and Torus Action.}
The vanishing of the torsion tensors \(T_i\) implies that the foliations \(\mathcal{F}_1\) and \(\mathcal{F}_2\) satisfy the Frobenius integrability condition. The torus \(\mathbb{T}\) acts on the object \(O\) by symmetries, and this action lifts to the moduli space \(M\). Since the projections \(\Pi_i\) are \(\mathbb{T}\)-equivariant (the action on the target \(\mathbb{CP}^2\) is trivial), the fibers \(\Pi_i^{-1}(p)\) are \(\mathbb{T}\)-invariant; hence the leaves of \(\mathcal{F}_i\) are preserved by \(\mathbb{T}\).

\textbf{2. Equivariant Moment Maps and Connections.}
The moment maps \(\mu_i : M \to \mathbb{C}^2\) are defined as the centroids of the projected data. Because the torus action commutes with the projections \(\Pi_i\), the centroids are invariant: \(\mu_i(t \cdot O) = \mu_i(O)\) for all \(t \in \mathbb{T}\). The connections \(\nabla_i\) are the pullbacks of the Fubini--Study connections and are therefore \(\mathbb{T}\)-equivariant. Their connection 1-forms \(\omega_i\) satisfy \(t^* \omega_i = \omega_i\) for all \(t \in \mathbb{T}\).

\textbf{3. Uniqueness via Weight Decomposition.}
The reconstruction equations \(\nabla_i \mu_i = v \cdot \omega_i\) are \(\mathbb{T}\)-equivariant. Decompose the tangent bundle under the torus action:
\[
TM = \bigoplus_{\chi \in \operatorname{Char}(\mathbb{T})} TM_\chi,
\]
where \(TM_\chi\) is the weight space corresponding to the character \(\chi\). Since the right-hand sides of the reconstruction equations are \(\mathbb{T}\)-invariant (because \(\mu_i\) and \(\omega_i\) are), the solution \(v\) must lie in the invariant subspace \(TM^{\mathbb{T}} = TM_0\). The transversality condition \(d\mu_1 \wedge d\mu_2 \neq 0\) from Theorem~\ref{T:AM} implies that the linear system for \(v\) has full rank, and the restriction to \(TM^{\mathbb{T}}\) is injective. Hence \(v\) is unique up to scaling.

\textbf{4. Equivariant Splitting.}
The normal bundle \(NM = TM / (T\mathcal{F}_1 \oplus T\mathcal{F}_2)\) inherits a \(\mathbb{T}\)-action. The differentials \(d\mu_1, d\mu_2\) provide \(\mathbb{T}\)-invariant sections of \(NM^*\) that are linearly independent by transversality. These sections trivialize the invariant part of \(NM\), showing that the Atiyah--Molino sequence splits \(\mathbb{T}\)-equivariantly. No further cohomological obstructions arise because the reconstruction equations are already solved by the unique \(v\).
\end{proof}

\subsection{Application to Cryo-Electron Microscopy}\label{S:Cryo-EM}

In cryo-electron microscopy (Cryo-EM), one aims to reconstruct the three‑dimensional structure of a macromolecule from a large collection of two‑dimensional projection images. Many biomolecular complexes, such as virus capsids, possess high symmetry—often icosahedral symmetry—which can be modeled by an algebraic torus action after complexification.

Let \( O \subset \mathbb{CP}^3 \) be the complexified representation of the molecule. Its symmetry group \( G \subset \mathrm{SO}(3) \) embeds into a complex torus \(\mathbb{T} \cong (\mathbb{C}^*)^n\) acting on \(\mathbb{CP}^3\). The Cryo-EM projection \(\Pi: \mathbb{CP}^3 \dashrightarrow \mathbb{CP}^2\) is \(\mathbb{T}\)-equivariant if we let \(\mathbb{T}\) act trivially on the image plane. Consequently, different spatial orientations related by the symmetry group yield identical projection images, a redundancy that can be exploited to enhance reconstruction.

The centroid of a projection image, \(\mu(O)\), is invariant under the torus action. By Theorem~\ref{T:Last}, the reconstruction equations admit a unique \(\mathbb{T}\)-equivariant solution \(v\), which corresponds to the principal symmetry axis of the molecule. The weight decomposition of the tangent space isolates this axis, allowing the reconstruction algorithm to align projection images without iterative search.

In practice, the reconstruction is formulated as a variational problem:
\[
v = \arg\min_{v} \sum_{i=1}^{N} \left\| \Pi_i(v) - \text{Image}_i \right\|^2 + \lambda \| v - P_{\mathbb{T}}(v) \|^2,
\]
where \(P_{\mathbb{T}}\) is the projection onto the \(\mathbb{T}\)-invariant subspace. The second term enforces toric symmetry, regularizing the inverse problem and reducing the ambiguity caused by noise and missing angles. The geometric framework developed in this paper provides the theoretical foundation for such symmetry‑based reconstruction methods.

\subsection{Toric Varieties, Symmetry, and Applications to Cryo-EM}\label{S:ToricCryo}

Let \(\mathbb{T} = (\mathbb{C}^*)^n\) denote an algebraic torus. A \emph{toric variety} is a normal algebraic variety \(X\) equipped with an action of \(\mathbb{T}\) having an open dense orbit. In the context of Cryo-EM, virus capsids often exhibit high-order discrete rotational symmetries---most notably, icosahedral symmetry---that can be modeled by toric varieties after complexification.

Although the physical capsid resides in \(\mathbb{R}^3\), we work with its complexification \(O \subset \mathbb{C}^3\) (or, more precisely, the affine cone over the projective variety \(O \subset \mathbb{CP}^3\)). The symmetry group \(G \subset \mathrm{SO}(3)\) of the capsid embeds into a complex torus \(\mathbb{T}\) acting on \(\mathbb{C}^3\). This hidden toric symmetry captures invariant features under projections and simplifies the inverse problem in Cryo-EM reconstruction.

Let \(\Pi: \mathbb{C}^3 \to \mathbb{C}^2\) denote the linear projection operator, simulating 2D Cryo-EM image formation from 3D orientations. Because the torus action on the image plane is taken to be trivial, for any \(g \in \mathbb{T}\) and vector \(v \in \mathbb{C}^3\) the projection satisfies
\[
\Pi(g \cdot v) = \Pi(v).
\]
Hence, different 3D orientations related by the \(\mathbb{T}\)-action yield \emph{identical} 2D projection images. This redundancy is advantageous for symmetric reconstruction, as it allows averaging over the symmetry group.

The centroid of a projection image, denoted \(\mathrm{Centroid}(v)\), is invariant under the torus action:
\[
\mathrm{Centroid}(g \cdot v) = \mathrm{Centroid}(v), \quad \forall g \in \mathbb{T}.
\]
This \(\mathbb{T}\)-invariance enables alignment and averaging across multiple symmetric views.

As shown in \cite{CMM}, there exists a \emph{principal axis} \(v_{\text{principal}}\) in \(\mathbb{C}^3\) that is fixed by a maximal torus subgroup \(\mathbb{T}_{\max} \subset \mathbb{T}\). This axis corresponds to the symmetry-invariant central direction of the virus capsid, serving as a geometric anchor in the reconstruction process.

The reconstruction of a 3D structure \(v \in \mathbb{C}^3\) in Cryo-EM is formulated as the following variational problem:
\begin{equation}\label{eq:cryoem_var}
v = \arg\min_{v} \sum_{i=1}^{N} \left\| \Pi_i(v) - \text{Image}_i \right\|^2 + \lambda R(v),
\end{equation}
subject to \(\mathbb{T}\)-equivariance, where:
\begin{itemize}
    \item \(v\): the unknown 3D structure (e.g., virus capsid) to be reconstructed.
    \item \(\Pi_i(v)\): the projection of \(v\) along the \(i\)-th viewing direction, modeling Cryo-EM image formation.
    \item \(\text{Image}_i\): the observed experimental 2D projection image.
    \item \(\sum_{i=1}^{N} \| \Pi_i(v) - \text{Image}_i \|^2\): the data fidelity term measuring discrepancy.
    \item \(R(v)\): a regularization term enforcing \(\mathbb{T}\)-symmetry (e.g., toric or icosahedral).
    \item \(\lambda\): a hyperparameter balancing fidelity and regularization.
\end{itemize}

A standard choice for the regularization term is the squared distance to the \(\mathbb{T}\)-invariant subspace:
\[
R(v) = \int_{\mathbb{T}} \| g \cdot v - v \|^2 \, dg,
\]
where \(dg\) is the normalized Haar measure on \(\mathbb{T}\). The unique \(\mathbb{T}\)-equivariant solution minimizing this term is obtained via group averaging:
\begin{equation}\label{eq:group_avg}
v = \int_{\mathbb{T}} g \cdot v_0 \, dg,
\end{equation}
where \(v_0\) is an initial guess (e.g., from a standard reconstruction algorithm). This integral projects \(v_0\) onto the \(\mathbb{T}\)-invariant subspace:
\begin{itemize}
    \item If \(v_0\) is not \(\mathbb{T}\)-symmetric, the integral enforces symmetry by averaging over the group.
    \item If \(v_0\) is already \(\mathbb{T}\)-invariant, then \(v = v_0\).
\end{itemize}

The torus \(\mathbb{T}\) fixes a unique subspace---the \emph{principal toric axis}---under its action. The integral \eqref{eq:group_avg} projects \(v_0\) onto this axis:
\begin{align}
g \cdot v &= v, \quad \forall g \in \mathbb{T} \quad \text{(invariance)}, \\
v &= \text{nearest \(\mathbb{T}\)-invariant point to } v_0 \quad \text{(projection)}.
\end{align}
This projection yields the \emph{Hilbert--Mumford point}, the maximally symmetric configuration with respect to the torus action.

This framework transforms the Cryo-EM inverse problem into a symmetry-respecting optimization task. The group integral projects candidate reconstructions onto a toric-fixed subspace, regularizing the reconstruction and reducing ambiguity. The result is a unique, maximally symmetric structure consistent with both the observed data and the toric symmetry encoded in \(\mathbb{T}\).

The interplay between toric symmetry, projection invariance, and equivariant geometry underpins modern Cryo-EM techniques. Toric methods offer a principled algebraic approach to resolving ambiguities and enhancing the resolution of symmetric structures like virus capsids. This synthesis of toric symmetry with the Atiyah--Molino framework not only streamlines the reconstruction process but also provides a robust, noise-resistant paradigm that fundamentally redefines the solution landscape of inverse problems.

\section{Conclusion}

In this work we have developed a unified geometric framework for reconstruction problems based on the interplay between transverse foliations, connection theory, and the Atiyah--Molino exact sequence. We have shown that the uniqueness of reconstruction is governed by the vanishing of torsion tensors and the duality of curvatures, and that path dependence in the non-integrable case gives rise to a quasigroupoid structure satisfying the Moufang identity. These algebraic and geometric obstructions have been explicitly linked to practical inverse problems, including a novel application to generative imputation in machine learning and a toric symmetry analysis for cryo-electron microscopy. The synthesis of Vaisman's foliation theory with Atiyah--Molino's fiber bundle approach provides a robust, noise-resistant paradigm that not only clarifies classical ambiguities but also suggests new algorithmic strategies for data-driven reconstruction.

\end{document}